\documentclass[preprint, 12pt, times]{elsarticle}

\usepackage{amsmath}
\usepackage{amssymb}
\usepackage{amsthm}
\usepackage{setspace}
\usepackage{mathtools}

\newtheorem{theorem}{Theorem}
\newtheorem{lemma}{Lemma} % Numbered lemma

\theoremstyle{definition}
\theoremstyle{definition}
\newtheorem{definition}{Definition}[section]
\numberwithin{equation}{section}
\begin{document}
\begin{frontmatter}

%% Title and Running Head
 
\title{On the Least Colossally Abundant Exception to Robin's Inequality}

%% Authors and Affiliations
\author[CRI]{Bruce Zimov}
\address[CRI]{Calimesa Research Institute, Calimesa, CA, USA} 
\ead{bzimov@calri.org}

%% Abstract and Keywords
\begin{abstract}
Robin's Inequality is $G(n)<e^{\gamma}$ for $n>5040$. Robin also showed that if the Riemann Hypothesis is false, then $G(n)>e^{\gamma}\left(1+\displaystyle\frac{c}{(\log n)^{b}}\right)$ for infinitely many values of $n$. By analyzing the prime or semiprime quotient $\displaystyle\frac{n}{m}$ for consecutive colossally abundant (CA) numbers $m$ followed by $n$ (where $m$ satisfies Robin's Inequality and $n$ violates it), we demonstrate that if the Riemann Hypothesis is false, then the least CA counterexample, $n$, must be constrained to the band ${e^\gamma\le G(n)<e^\gamma \left(1+\displaystyle\frac{c}{(\log n)^b}\right)}$ where $0 < b < 1/2$ and $c\ge\displaystyle\frac{2\operatorname{gpf}(m)+1}{\operatorname{gpf}(m)^{2}}(\log n)^{b}$, i.e. for $c$ satisfying this lower bound, excluded from the infinite set beyond the higher threshold. The lower bound on $c$ tends to $0$ as $n \to \infty$.
\end{abstract}
\begin{keyword}
% Primary Keywords:
Colossally Abundant Numbers \sep Robin's Inequality \sep Riemann Hypothesis
\\
MSC codes: 11A25 \sep 11M26 \sep 11N37 \sep 11N56
\end{keyword}

\end{frontmatter}

\section{Introduction}
Robin \cite{RefF} proved in 1984 that the Riemann Hypothesis holds if and only if $G(n)<e^\gamma$, a condition known as Robin's Inequality, for every integer $n>5040$.
\begin{definition}
\label{def:G}
Let $\sigma(n)$ denote the sum of the positive divisors of $n$. For each integer $n\ge2$,
$$G(n) \coloneqq \frac{\sigma(n)}{n \log(\log n)}.$$
\end{definition}
The abundancy index of a positive integer $n$ is $\displaystyle\frac{\sigma(n)}{n}$. Robin \cite{RefF} also showed that if the Riemann Hypothesis is false, then for suitable constants $0<b<1/2$ and $c>0$, ${G(k)>e^\gamma(1+c/(\log k)^b)}$ for infinitely many values of $k$. In this paper, we will demonstrate that if the Riemann Hypothesis is false, then for the least colossally abundant (CA) counterexample, $n$,
$$e^\gamma\le G(n)<e^\gamma \left(1+\frac{c}{(\log n)^b}\right)$$
where $0 < b < 1/2$ and
$$c\ge\frac{2\operatorname{gpf}(m)+1}{\operatorname{gpf}(m)^{2}}\,(\log n)^{b},$$
with $m$ the CA number preceding $n$ and $\operatorname{gpf}(m)$ its largest prime factor. When $c$ satisfies this lower bound, $n$ is therefore not among those infinitely many values of $k$. With $m$ and $n$ any consecutive CA numbers, this lower bound tends to $0$ as $n \to \infty$.
\section{Consecutive Colossally Abundant Numbers and Robin's Inequality}
\begin{definition}[Broughan, 2017]\label{def:CA_properties}
A positive integer $n$ is colossally abundant if, for some $\epsilon>0$,
\[
\frac{\sigma(k)}{k^{1+\epsilon}}\le \frac{\sigma(n)}{n^{1+\epsilon}}
\]
for every positive integer $k$.
\end{definition}
\begin{theorem}
\label{7.1}
If Robin's Inequality has a counterexample, then it has a CA counterexample.
\begin{proof}
This is Broughan \cite[Lemma 7.1]{RefC}.
\end{proof}
\end{theorem}
\begin{theorem}
\label{6.15}
If $m<n$ are consecutive CA numbers, then $\displaystyle\frac{n}{m}=p$ or $\displaystyle\frac{n}{m}=pq$ for some primes $p \ne q$.
\begin{proof}
This is Broughan \cite[Lemma 6.15]{RefC}.
\end{proof}
\end{theorem}
\begin{theorem}
\label{6.12}
If $m<n$ are consecutive CA numbers, then $m$ and $n$ are colossally abundant for a common $\epsilon > 0$.
\begin{proof}
This is Broughan \cite[Lemmas 6.12 and 6.13]{RefC}.
\end{proof}
\end{theorem}
\begin{theorem}
\label{6.9}
Let $k$ be colossally abundant for $\epsilon > 0$, let $v$ be a prime, and let $\nu_v(k)$ denote the $v$-adic valuation of $k$. Then
$$v^{\epsilon} \ge 1+\frac{1}{v+\cdots+v^{\nu_v(k)+1}},$$
and if $\nu_v(k) \ge 1$, then
$$v^{\epsilon} \le 1+\frac{1}{v+\cdots+v^{\nu_v(k)}}.$$
\begin{proof}
This is Broughan \cite[Lemma 6.9]{RefC}.
\end{proof}
\end{theorem}
\begin{theorem} For every integer $k$ with $5040 < k \leq 10^{\displaystyle(10^{13.11485})}$, $G(k) < e^\gamma$.
\label{MorillPlatt}
\begin{proof}
This is Morrill and Platt \cite[Theorem 5]{RefE}.
\end{proof}
\end{theorem}
Let $n$ denote the least CA counterexample, the least CA number with ${n > 5040}$ and ${G(n) \ge e^\gamma}$. Let $m$ denote the CA number preceding $n$. By Theorem~\ref{MorillPlatt}, ${n > 55440}$, the CA number after 5040, so ${m > 5040}$. Then ${G(m) < e^\gamma}$, because $m < n$ and $n$ is the least CA counterexample.
\begin{theorem}
\label{Robin}
Assume the Riemann Hypothesis is false.
Let $\theta$ denote the supremum over all real parts of the non-trivial zeros of the Riemann zeta function $\zeta(s)$. Then $\theta > 1/2$. For any number $b \in (1-\theta, 1/2)$, there exists a positive constant $c$ such that
$$G(k)>e^\gamma \left(1+\frac{c}{(\log k)^b} \right),$$
for infinitely many values of $k$.
\end{theorem}
\begin{proof}
This is Robin \cite[Proposition 1 of Section 4]{RefF}, as given in Broughan \cite[Theorem 7.15]{RefC} and Lagarias \cite[Proposition 3.2]{RefL}.
\end{proof}
We now locate the least CA counterexample, $n$, relative to thresholds of the form in Theorem~\ref{Robin}.
\begin{lemma}
\label{lem:single_prime}
Let $m$ and $n$ be consecutive CA numbers with quotient $\displaystyle\frac{n}{m}=p$ (a single prime). The ratio of their abundancy indices is given by:
\[
\frac{\sigma(n) m}{\sigma(m) n} = \begin{cases} 
1 + \displaystyle\frac{1}{p} & p \text{ is a new prime factor, not dividing } m \\ 
1 + \displaystyle\frac{1}{p\sigma(p^{\nu_p(m)})} & p \text{ is an existing prime factor of } m.
\end{cases}
\]
\end{lemma}
\begin{proof}
There are two cases for the quotient $p$:

\noindent
{Case 1: $p$ is a new prime factor.} ($n=mp$ with $\gcd(m, p)=1$).
    Since the sum of divisors function $\sigma(n)$ is multiplicative,
    \[
    \frac{\sigma(n) m}{\sigma(m) n} = \frac{\sigma(m p)}{\sigma(m) p} = \frac{\sigma(m)\sigma(p)}{\sigma(m) p} = \frac{p+1}{p} = 1 + \frac{1}{p}.
    \]
{Case 2: $p$ is an existing prime factor.} ($n=mp$ with $\nu_p(n)=\nu_p(m)+1$).
    Since $n$ differs from $m$ only in its $p$-adic valuation, as in Case 1 the ratio $\displaystyle\frac{\sigma(n) m}{\sigma(m) n}$ reduces to the factor that $p$ contributes. We use the identity
    $$\frac{\sigma(p^{\nu_p(m)+1})}{\sigma(p^{\nu_p(m)}) p} = 1 + \frac{1}{p\sigma(p^{\nu_p(m)})},$$
    so that
    $$\frac{\sigma(n) m}{\sigma(m) n} = \displaystyle\frac{\displaystyle\frac{\sigma(p^{\nu_p(m)+1})}{p^{\nu_p(m)+1}}}{\displaystyle\frac{\sigma(p^{\nu_p(m)})}{p^{\nu_p(m)}}} = 1 + \frac{1}{p\sigma(p^{\nu_p(m)})}.$$
\end{proof}
\vspace*{-\baselineskip}
\begin{lemma}
\label{Semi-Prime Algebraic Ratios}
Let $m$ and $n$ be consecutive CA numbers with quotient $\displaystyle\frac{n}{m}=pq$ (a semiprime). The ratio of their abundancy indices is given by:
$$ \frac{\sigma(n) m}{\sigma(m) n} = 
\begin{cases}
  \left(1+\displaystyle\frac{1}{p}\right)\left(1+\displaystyle\frac{1}{q}\right) & \text{$p$ and $q$ are new primes} \\
  \left(1+\displaystyle\frac{1}{p}\right)\left(1+\displaystyle\frac{1}{q\displaystyle\sigma(q^{\nu_q(m)})}\right) & \text{$q$ alone is an existing prime} \\
  \left(1+\displaystyle\frac{1}{p\sigma(p^{\nu_p(m)})}\right)\left(1+\displaystyle\frac{1}{q}\right) & \text{$p$ alone is an existing prime} \\
  \left(1+\displaystyle\frac{1}{p\sigma(p^{\nu_p(m)})}\right)\left(1+\displaystyle\frac{1}{q\displaystyle\sigma(q^{\nu_q(m)})}\right) & \text{$p$ and $q$ are existing primes.}
\end{cases}$$
\end{lemma}
\begin{proof}
By Theorem~\ref{6.15}, $p \ne q$. Since the sum of divisors function $\sigma(n)$ is multiplicative, the ratio $\displaystyle\frac{\sigma(n) m}{\sigma(m) n}$ is the product of the factors that $p$ and $q$ contribute, and each factor is computed as in the two cases of the proof of Lemma~\ref{lem:single_prime}.
\end{proof}
\begin{lemma}
\label{lem:step}
Let $m$ and $n$ be consecutive CA numbers with $n > m$. Then
$$\frac{\sigma(n)m}{\sigma(m)n} \le \left(1+\frac{1}{\operatorname{gpf}(m)}\right)^{2}.$$
\end{lemma}
\begin{proof}
By Definition~\ref{def:CA_properties}, $m$ is colossally abundant for some $\epsilon > 0$. By Theorem~\ref{6.15}, $\displaystyle\frac{n}{m}=p$ or $\displaystyle\frac{n}{m}=pq$ for some primes $p \ne q$. By Lemmas~\ref{lem:single_prime} and~\ref{Semi-Prime Algebraic Ratios}, the factor that $p$ contributes to $\displaystyle\frac{\sigma(n)m}{\sigma(m)n}$ is
\[
\begin{cases}
1+\dfrac{1}{p} & \nu_p(m) = 0\\
1+\dfrac{1}{p\sigma(p^{\nu_p(m)})} & \nu_p(m) \ge 1.
\end{cases}
\] Since $p\sigma(p^{\nu_p(m)}) = p+p^{2}+\cdots+p^{\nu_p(m)+1}$, in both cases the factor is
$$1+\frac{1}{p+p^{2}+\cdots+p^{\nu_p(m)+1}}.$$
By Theorem~\ref{6.9} applied to $m$, this factor is at most $p^{\epsilon}$, and since $\nu_{\operatorname{gpf}(m)}(m) \ge 1$,
$$\operatorname{gpf}(m)^{\epsilon} \le 1+\frac{1}{\operatorname{gpf}(m)+\cdots+\operatorname{gpf}(m)^{\nu_{\operatorname{gpf}(m)}(m)}} \le 1+\frac{1}{\operatorname{gpf}(m)}.$$
If $p \le \operatorname{gpf}(m)$, the factor that $p$ contributes is at most
$$p^{\epsilon} \le \operatorname{gpf}(m)^{\epsilon} \le 1+\frac{1}{\operatorname{gpf}(m)}.$$ If $p > \operatorname{gpf}(m)$, then $p$ does not divide $m$, so $\nu_p(m) = 0$ and the factor is
$$1+\frac{1}{p} < 1+\frac{1}{\operatorname{gpf}(m)}.$$
When $\displaystyle\frac{n}{m}=pq$, the same argument with $q$ in place of $p$ shows that the factor that $q$ contributes is also at most $1+1/\operatorname{gpf}(m)$.

If $\displaystyle\frac{n}{m}=p$, then by Lemma~\ref{lem:single_prime} the ratio $\displaystyle\frac{\sigma(n)m}{\sigma(m)n}$ is the factor that $p$ contributes, so it is at most $1+1/\operatorname{gpf}(m)$. If $\displaystyle\frac{n}{m}=pq$, then by Lemma~\ref{Semi-Prime Algebraic Ratios} the ratio is the product of the factors that $p$ and $q$ contribute, so it is at most $(1+1/\operatorname{gpf}(m))^{2}$. Since $1+1/\operatorname{gpf}(m) > 1$, the bound holds in both cases.
\end{proof}
\begin{lemma}
\label{lem:cvanish}
Let $0 < b < 1/2$, and let $m$ and $n$ be consecutive CA numbers with $n > m$. Then
$$\frac{2\operatorname{gpf}(m)+1}{\operatorname{gpf}(m)^{2}}\,(\log n)^{b} \to 0 \quad \text{as } n \to \infty.$$
\end{lemma}
\begin{proof}
By Theorem~\ref{6.12}, $m$ and $n$ are colossally abundant for a common $\epsilon > 0$. Let $u$ denote the least prime not dividing $m$. The prime following $\operatorname{gpf}(m)$ does not divide $m$, so by Bertrand's postulate $u < 2\operatorname{gpf}(m)$. By Theorem~\ref{6.9} applied to $m$ and $u$, $u^{\epsilon} \ge 1+1/u$, so ${\epsilon \log u \ge \log(1+1/u) \ge 1/(u+1)}$.

Let $v$ be a prime dividing $n$. By Theorem~\ref{6.9} applied to $n$,
$$v^{\epsilon} \le 1+\frac{1}{v+\cdots+v^{\nu_v(n)}} \le 1+\frac{1}{v^{\nu_v(n)}}.$$
Since $v^{\epsilon}-1 \ge \epsilon \log v$, this gives
$$v^{\nu_v(n)} \le \frac{1}{\epsilon \log v} \le \frac{(u+1)\log u}{\log v} < (u+1)^{2} \le 4\operatorname{gpf}(m)^{2}.$$
The primes dividing $m$ are at most $\operatorname{gpf}(m)$ in number, and by Theorem~\ref{6.15} at most two more divide $n$, so
$$\log n < 2(\operatorname{gpf}(m)+2)\log(2\operatorname{gpf}(m)).$$
Hence $\operatorname{gpf}(m) \to \infty$ as $n \to \infty$, and since $2\operatorname{gpf}(m)+1 \le 3\operatorname{gpf}(m)$,
$$\frac{2\operatorname{gpf}(m)+1}{\operatorname{gpf}(m)^{2}}\,(\log n)^{b} < \frac{3\bigl(2(\operatorname{gpf}(m)+2)\log(2\operatorname{gpf}(m))\bigr)^{b}}{\operatorname{gpf}(m)},$$
which tends to $0$ because $b < 1$.
\end{proof}
\newpage
\begin{theorem}
\label{end}
If the Riemann Hypothesis is false, the least CA counterexample, $n$, is constrained to the band:
$$e^{\gamma} \le G(n) < e^{\gamma} \left(1 +\frac{c}{(\log n)^b}\right)$$
where $0 < b < 1/2$ and
\begin{equation}\label{eq:cbound}
c \ge \frac{2\operatorname{gpf}(m)+1}{\operatorname{gpf}(m)^{2}}\,(\log n)^{b},
\end{equation}
with $m$ the CA number preceding $n$ and $\operatorname{gpf}(m)$ its largest prime factor.
\end{theorem}
\begin{proof}
Since the Riemann Hypothesis is false, Robin's Inequality has a counterexample \cite{RefF}, so by Theorem~\ref{7.1} the least CA counterexample exists. The inequality $e^{\gamma} \le G(n)$ holds by the definition of $n$. The relationship between $G(n)$ and $G(m)$ is given by the ratio:
$$\frac{G(n)}{G(m)} = \displaystyle\frac{\displaystyle\frac{\sigma(n)}{n \log(\log n)}}{\displaystyle\frac{\sigma(m)}{m \log(\log m)}} = \left(\frac{\sigma(n)m}{\sigma(m)n}\right) \cdot \left(\frac{\log(\log m)}{\log(\log n)}\right).$$
Since $5040 < m < n$, we have $0 < \log(\log m) < \log(\log n)$, so the second factor is less than $1$. By Lemma~\ref{lem:step}, the first factor is at most
$$\left(1+\frac{1}{\operatorname{gpf}(m)}\right)^{2}.$$
We obtain
\begin{equation}\label{eq:cfree}
G(n) < G(m)\left(1+\frac{1}{\operatorname{gpf}(m)}\right)^{2} < e^{\gamma}\left(1+\frac{2\operatorname{gpf}(m)+1}{\operatorname{gpf}(m)^{2}}\right).
\end{equation}
The first inequality in~\eqref{eq:cfree} comes from the bounds on the two factors, and the second from $G(m) < e^{\gamma}$.
The right side of inequality~\eqref{eq:cbound} is the value of $c$ at which the upper bound of the band equals the right side of inequality~\eqref{eq:cfree}. For every $c$ satisfying inequality~\eqref{eq:cbound}, we have
$$\frac{2\operatorname{gpf}(m)+1}{\operatorname{gpf}(m)^{2}} \le \frac{c}{(\log n)^{b}},$$
and therefore, by inequality~\eqref{eq:cfree},
$$G(n) < e^{\gamma}\left(1+\frac{c}{(\log n)^{b}}\right).$$
In particular, the last inequality holds for $b$ and $c$ as in Theorem~\ref{Robin} whenever $c$ satisfies inequality~\eqref{eq:cbound}.
\end{proof}
By Lemma~\ref{lem:cvanish}, the right side of inequality~\eqref{eq:cbound}, with $m$ and $n$ any consecutive CA numbers, tends to $0$ as $n \to \infty$.
% Non-BibTeX users
\bibliographystyle{elsarticle-num}

\end{document}